\documentclass[a4paper,12pt,reqno]{amsart}
\usepackage{graphicx}
\allowdisplaybreaks
\usepackage[colorlinks=true,linkcolor=magenta,citecolor=blue,urlcolor=cyan]{hyperref}
\usepackage{marvosym}
\theoremstyle{plain}
\newtheorem{theorem}{Theorem}[section]

\newtheorem{lemma}[theorem]{Lemma}

\theoremstyle{definition}

\theoremstyle{remark}
\newtheorem{remark}[theorem]{Remark}

\begin{document}
	
	\title[Vanishing Coefficients of $q^{5n+r}$ and $q^{7n+r}$ in Certain Infinite $q$-series Expansions]
	{Vanishing Coefficients of $q^{5n+r}$ and $q^{7n+r}$ in Certain Infinite $q$-series Expansions }
	
		\author[Thejitha M.P ]{Thejitha M.P}
	\address{Thejitha M.P\\Ramanujan School of Mathematical Sciences\\Department of Mathematics\\ Pondicherry University\\ Puducherry 605 014, India}
	\email{tthejithamp@gmail.com}
		\author[ Anusree Anand]{ Anusree Anand}
	\address{Anusree Anand \\Ramanujan School of Mathematical Sciences\\Department of Mathematics\\ Pondicherry University\\ Puducherry 605 014, India}
	\email{anusreeanand05@gmail.com}
	
	\author[Fathima S.N]{Fathima S.N}
	\address{Fathima S. N \\Ramanujan School of Mathematical Sciences\\Department of Mathematics\\ Pondicherry University\\ Puducherry 605 014, India}
	\email{dr.fathima.sn@pondiuni.ac.in   (\Letter)}

	
	\begin{abstract}
			Motivated by the recent work of several authors on vanishing coefficients of the arithmetic progression in certain $q$-series expansion, we study some variants of these $q$-series and prove some comparable results. For instance, if
			\begin{align*}
				\sum_{n=0}^{\infty}c_1(n)q^n=\left(\pm q^2,\pm q^3; q^5\right)_\infty^2  \left( q, q^{14}; q^{15}\right)_\infty,
				\end{align*}
			then $c_1(5n+3)=0$.
	\end{abstract}
	
	\maketitle
	
	\section{Introduction}
	For complex numbers $a$ and $q$, with $|q|<1$, we define
	\begin{align*}
		(a;q)_\infty:=\prod_{k=0}^{\infty}\left(1-aq^k\right)
	\end{align*}
	and
	\begin{align*}
		\left(a_1,a_2,\cdots a_n;q\right)_\infty:=(a_1;q)_\infty (a_2;q)_\infty\cdots (a_n;q)_\infty.
	\end{align*}
	In \cite{a5}, Hirschhorn studied the following two $q$-series 
	\begin{align}\label{eq:v.1.1}
		(-q,-q^4;q^5)_\infty (q,q^9;q^{10})_\infty^3=\sum_{n=0}^{\infty}a(n)q^n
	\end{align}
	and 
	\begin{align}\label{eq:v.1.2}
		(-q^2,-q^3;q^5)_\infty (q^3,q^7;q^{10})_\infty^3= \sum_{n=0}^{\infty}b(n)q^n. 
	\end{align}
	He proved that 
	\begin{align*}
		a(5n+2)=a(5n+4)=0
	\end{align*} and
	\begin{align*}
		b(5n+1)=b(5n+4)=0.
	\end{align*}
	In sequel to the work of Hirschhorn, Tang \cite{a9} further investigated the vanishing coefficients of the arithmetic progression in infinite products similar to those defined in \eqref{eq:v.1.1} and \eqref{eq:v.1.2}. In particular, $a_1(n)$ and $b_1(n)$ are defined by 
	\begin{align*}
		(-q,-q^4;q^5)_\infty^3 (q^3,q^7;q^{10})_\infty=\sum_{n=0}^{\infty}a_1(n)q^n
	\end{align*} and
	\begin{align*}
		(-q^2,-q^3;q^5)_\infty^3 (q,q^9;q^{10})_\infty=\sum_{n=0}^{\infty}b_1(n)q^n.
	\end{align*}
	Tang proved that for $n\geq 0$, 
	\begin{align*}
		a_1(5n+3)=b_1(5n+1)=0.
	\end{align*}
	The properties of coefficients in power series expansions of various infinite products or quotients of infinite products have been extensively investigated. For more details, see \cite{a3,a6,a8,a9}. Further, exploration in the realm of vanishing coefficients in infinite products was made by Mc Laughlin\cite{a12} and Tang\cite{a10,a11} by considering a more general form of infinite products.
	
	Very recently, Ananya et. al {{\cite[Theorem 1.3]{anan}}}, subjected to the conditions in their paper, obtained the following results for vanishing coefficients of families of infinite products:
\begin{align}
X_{t,2t,5\ell,15\ell,2,1}(5n+2t)=Z_{t,2t,5\ell,15\ell,2,1}(5n+2t)=0,\label{3}\\
X_{2t,t,7\ell,21\ell,1,2}(7n+2t)=Y_{2t,t,7\ell,21\ell,1,2}(7n+2t)=0,\label{4}\\
X_{t,6t,7\ell,21\ell,2,1}(7n+4t)=Z_{t,6t,7\ell,21\ell,2,1}(7n+4t)=0\label{5}.
\end{align}
	\noindent In this paper, we further obtain new vanishing coefficients of arithmetic progression modulo 5 and 7 for certain infinite products. Our proof employs q-series manipulations, Jacobi triple product identity and elementary Ramanujan's theta function identities.
	
	\section{Preliminaries}
	
	 In this section, we review some preliminary results and lemmas that will be used in the subsequent section. 
	 	Ramanujan's general theta function \cite[Chapter 16, Eq.(18.1)]{a4}  is defined as follows:
	 \begin{align*}
	 	f(a,b):=\sum_{k=-\infty}^{\infty}a^{k(k+1)/2}b^{k(k-1)/2},\rm {where}\hskip.2cm|ab|<1.
	 \end{align*}
	 The Jacobi triple product identity in terms of Ramanujan's theta function is given by
	 \begin{align*}
	 	f(a,b)=\left(-a,-b, ab; ab\right)_\infty.
	 \end{align*}
	 We recall the following basic properties satisfied by $f(a,b)$, which we frequently use in our proofs without explicitly mentioning them.
	 \begin{lemma}\label{l.1} \cite[Chapter 16, entry 19]{a4} We have
	 	\begin{align*}
	 		f(a,b)=f(b,a),
	 	\end{align*}
	 	\begin{align*}
	 		f(1,a)=2f(a,a^3).
	 	\end{align*}
	 \end{lemma}

	 \begin{lemma}\label{l.2}\cite[Chapter 16, entry 21]{a4} If $|q|<1$, then
	 	\begin{align*}
	 		\phi(q):=f(q,q)=\sum_{k=-\infty}^{\infty}q^{k^2},	\end{align*}
	 	\begin{align*}
	 		\psi(q):=f(q,q^3)=\sum_{k=0}^{\infty}q^{k(k+1)/2}.
	 	\end{align*}
	 \end{lemma}
	 \begin{lemma}\label{l.3}We have
	 	\begin{align*}
	 		f(a,b)=f(a^3b,ab^3)+af(b/a,a^5b^3),
	 	\end{align*}
	 	\begin{align*}
	 		f^2(a,b)=f(a^2,b^2)f(ab,ab)+af(b/a,a^3b)f(1,a^2b^2).
	 	\end{align*}
	 	Moreover if $ab=cd$, then
	 	\begin{equation}\label{1.2.9}
	 		f\left(a,b\right)f\left(c,d\right)= f\left(ac,bd\right)f\left(ad,bc\right)+af\left(\dfrac{b}{c}, \dfrac{c}{b}abcd\right)f\left(\dfrac{b}{d}, \dfrac{d}{b} abcd\right).
	 	\end{equation}
	 	\begin{proof}: Equation \eqref{1.2.9} comes from \cite[p.45, Entry 29]{a4} and \cite[p.9, Theorem 0.6]{a13}.
	 	\end{proof}
	 \end{lemma}

	\section{Vanishing Coefficients in Certain Infinite $q$-series Expansions}

In this section we state and prove our main results.

\begin{theorem}\label{t:v.1.1}
	For non-negative integers $\alpha$ and $\beta\leq 15$, define
	\begin{align*}
		\sum_{n=0}^{\infty}a(n)q^n=\left(q,q^4;q^5\right)_\infty \left(\pm q^6,\pm q^9;q^{15}\right)_\infty^2,
	\end{align*}
	\begin{align*}
		\sum_{n=0}^{\infty}b_\alpha(n)q^n=\left(\pm q,\pm q^4;q^5\right)_\infty^2 \left(q^\alpha,q^{15-\alpha};q^{15}\right)_\infty,
	\end{align*}
	\begin{align*}
		\sum_{n=0}^{\infty}c_\beta(n)q^n=\left(\pm q^2,\pm q^3;q^5\right)_\infty^2 \left(q^\beta,q^{15-\beta};q^{15}\right)_\infty.
	\end{align*}
	\rm{Then}\\
	$a(5n+3)=0,$\\
	$b_2(5n+2)=b_3(5n+4)=b_7(5n+3)=0,$\\
	$c_1(5n+3)=c_4(5n+4)=c_6(5n+1)=0.$                 
\end{theorem}
\begin{proof}[\it{Proof}] We start with\\
	$\displaystyle \sum_{n=0}^{\infty}a(n)q^n=\left(q,q^4;q^5\right)_\infty \left(q^6,q^9;q^{15}\right)_\infty^2$\\
	$=\dfrac{f\left(-q,-q^4\right)f\left(-q^6,-q^9\right)^2}{\left(q^5;q^5\right)_\infty\left(q^{15};q^{15}\right)_\infty^2}$\\
	$=\dfrac{\left(f\left(q^7,q^{13}\right)-qf\left(q^3,q^{17}\right)\right)\left(f\left(q^{12},q^{18}\right)\phi(q^{15})-2q^6f\left(q^3,q^{27}\right)\psi(q^{30})\right)}{\left(q^5;q^5\right)_\infty\left(q^{15};q^{15}\right)_\infty^2}$\\
	$=\dfrac{\phi(q^{15})\left(f\left(q^7,q^{13}\right)f\left(q^{12},q^{18}\right)-qf\left(q^3,q^{17}\right)f\left(q^{12},q^{18}\right)\right)}{\left(q^5;q^5\right)_\infty\left(q^{15};q^{15}\right)_\infty^2}\;\\
	+\dfrac{2\psi(q^{30})\left(q^7f\left(q^3,q^{17}\right)f\left(q^{3},q^{27}\right)-q^6f\left(q^7,q^{13}\right)f\left(q^{3},q^{27}\right)\right)}{\left(q^5;q^5\right)_\infty\left(q^{15};q^{15}\right)_\infty^2}$\\
	$=\dfrac{\phi(q^{15})(S_1-S_2)}{\left(q^5;q^5\right)_\infty\left(q^{15};q^{15}\right)_\infty^2}\;+\;\dfrac{2\psi(q^{15})(S_3-S_4)}{\left(q^5;q^5\right)_\infty\left(q^{15};q^{15}\right)_\infty^2}$,	\\
	\noindent
	where \begin{align*}
		S_1 = f(q^7,q^{13})f(q^{12},q^{18})=	\sum_{m,n=-\infty}^{\infty}q^{10m^2+3m+15n^2+3n},
	\end{align*}
	\begin{align*}
		S_2=qf(q^3,q^{17})f(q^{12},q^{18})=\sum_{m,n=-\infty}^{\infty}q^{10m^2+7m+15n^2+3n+1},
	\end{align*}
	\begin{align*}
		S_3=q^7f(q^3,q^{17})f(q^3,q^{27})=\sum_{m,n=-\infty}^{\infty}q^{10m^2+7m+15n^2+12n+7},
	\end{align*}
	\begin{align*}
		S_4=q^6f(q^7,q^{13})f(q^3,q^{27})=\sum_{m,n=-\infty}^{\infty}q^{10m^2+3m+15n^2+12n+6}.
	\end{align*}
	\noindent	In $S_1$, $3m+3n\equiv3\pmod5$ implies $m+n\equiv1\pmod5$ and $-2m+3n\equiv3\pmod5$. By solving these two congruence we get, $m=3s-r$ and $n=2s+r+1$. Therefore, 3-component of $S_1$ is \begin{align*}
		q^{18}\sum_{r,s=-\infty}^{\infty}q^{150s^2+25r^2+75s+30r}.
	\end{align*}
	In $S_2$, $7m+3n+1\equiv3\pmod5$ implies $2m+3n\equiv2\pmod5$ and $-m+n\equiv-1\pmod5$. By solving these two congruence we get, $m=s-3r+1$ and $n=2r+s$. Therefore, 3-component of 	$S_2$ is \begin{align*}
		q^{18}\sum_{r,s=-\infty}^{\infty}q^{150s^2+25r^2+75s+30r}.
	\end{align*}
	In $S_3$, $7m+12n\equiv-4\pmod5$ implies $m+n\equiv-2\pmod5$ and $2m-3n\equiv1\pmod5$. By solving these two congruence we get, $m=3s+r-1$ and $n=2s-r-1$. Therefore, 3-component of $S_3$ is \begin{align*}
		q^{13}\sum_{r,s=-\infty}^{\infty}q^{150s^2+25r^2-75s+5r}.
	\end{align*}	 	
	In $S_4$, $3m+12n\equiv-3\pmod5$ implies $m-n\equiv-1\pmod5$ and $-2m-3n\equiv2\pmod5$. By solving these two congruence we get, $m=3s-r-1$ and $n=-2s-r$. Therefore, 3-component of $S_4$ is \begin{align*}
		q^{13}\sum_{r,s=-\infty}^{\infty}q^{150s^2+25r^2-75s+5r}.
	\end{align*}
	Hence 3-components cancel in pairs and we get $a(5n+3)=0$. We exclude the proof of remaining identities as the proof is similar.
	This completes the proof of Theorem  \ref{t:v.1.1}.
\end{proof}
\begin{remark}
	We note, the identities	$b_2(5n+2)=b_7(5n+3)=c_4(5n+4)=c_6(5n+1)=0$, are also special cases of \eqref{3} with $\ell=1$ and $t=1,4,2,3$, respectively.
\end{remark}
	Let $i>0,j\geq0$ be integers and let $H(q)=\sum_{n=0}^{\infty}h(n)q^n$ be a formal power series. Define an operator $T_{i,j}$ by
\begin{align}
	T_{i,j}(H(q)):=\sum_{n=0}^{\infty}h(in+j)q^{in+j}.
\end{align}
\begin{theorem}\label{t:v.1.2}
	For non-negative integers $\alpha,\beta$ and $\gamma\leq 15$, define
	\begin{align*}
		\sum_{n=0}^{\infty}e_\alpha(n)q^n=\left(\pm q,\pm q^6;q^7\right)_\infty \left(q^\alpha,q^{21-\alpha};q^{21}\right)_\infty,
	\end{align*} 
	\begin{align*}
		\sum_{n=0}^{\infty}f_\beta(n)q^n=\left(\pm q^2,\pm q^5;q^7\right)_\infty \left(q^\beta,q^{21-\beta};q^{21}\right)_\infty,
	\end{align*} 
	\begin{align*}
		\sum_{n=0}^{\infty}g_\gamma(n)q^n=\left(\pm q^3,\pm q^4;q^7\right)_\infty \left(q^\gamma,q^{21-\gamma};q^{21}\right)_\infty.
	\end{align*} 
	\rm{Then} \\  $e_2(7n+5)=e_5(7n+6)=e_9(7n+2)=0,$\\
	$f_3(7n+5)=f_4(7n+3)=f_{10}(7n+4)=0,$\\
	$g_1(7n+4)=g_6(7n+1)=g_8(7n+6)=0.$
\end{theorem}
\begin{proof}[\it{Proof}] We start with\\
	$\displaystyle \sum_{n=0}^{\infty}e_2(n)q^n=\left(q,q^6;q^7\right)_\infty \left(q^2,q^{19};q^{21}\right)_\infty$\\
	$=\dfrac{f\left(-q,-q^6\right)f\left(-q^2,-q^{19}\right)}{\left(q^7;q^7\right)_\infty\left(q^{21};q^{21}\right)_\infty}$\\
	$=\dfrac{\left(f\left(q^9,q^{19}\right)-qf\left(q^5,q^{23}\right)\right)\left(f\left(q^{25},q^{59}\right)-q^2f\left(q^{17},q^{67}\right)\right)}{\left(q^7;q^7\right)_\infty\left(q^{21};q^{21}\right)_\infty}$\\
	$=\dfrac{\left(f\left(q^9,q^{19}\right)f\left(q^{25},q^{59}\right)-q^2f\left(q^9,q^{19}\right)f\left(q^{17},q^{67}\right)\right)}{\left(q^7;q^7\right)_\infty\left(q^{21};q^{21}\right)_\infty}\\
	-\dfrac{\left(qf\left(q^5,q^{23}\right)f\left(q^{25},q^{59}\right)-q^3f\left(q^5,q^{23}\right)f\left(q^{17},q^{67}\right)\right)}{\left(q^7;q^7\right)_\infty\left(q^{21};q^{21}\right)_\infty}$\\
	$=\dfrac{\left(S_1-S_2\right)-\left(S_3-S_4\right)}{\left(q^7;q^7\right)_\infty\left(q^{21};q^{21}\right)_\infty},$\\
	\noindent	where \begin{align*}
		S_1 = f(q^9,q^{19})f(q^{25},q^{59})=	\sum_{m,n=-\infty}^{\infty}q^{14m^2+5m+42n^2+17n},
	\end{align*}
	\begin{align*}
		S_2=q^2f(q^9,q^{19})f(q^{17},q^{67})=\sum_{m,n=-\infty}^{\infty}q^{14m^2+5m+42n^2+25n+2},
	\end{align*}
	\begin{align*}
		S_3=qf(q^5,q^{23})f(q^{25},q^{59})=\sum_{m,n=-\infty}^{\infty}q^{14m^2+9m+42n^2+17n+1},
	\end{align*}
	\begin{align*}
		S_4=q^3f(q^7,q^{13})f(q^3,q^{27})=\sum_{m,n=-\infty}^{\infty}q^{14m^2+9m+42n^2+25n+3}.
	\end{align*}
	In $S_1$, $5m+17n\equiv5\pmod7$ implies $-m-2n\equiv-1\pmod7$ and $-2m+3n\equiv-2\pmod7$. By solving these two congruence we get, $m=-3r-2s+1$ and $n=-2r+s$. Therefore, 5-component of $S_1$ is \begin{align*}
		T_{7,5}(S_1)= q^{19}\sum_{r,s=-\infty}^{\infty}q^{294r^2+98s^2-133r-49s}.
	\end{align*} 
	Similarly, we obtain
	\begin{align*}
		T_{7,5}(S_2)= q^{19}\sum_{r,s=-\infty}^{\infty}q^{294r^2+98s^2-133r+49s},
	\end{align*}
	\begin{align*}
		T_{7,5}(S_3)=	q^{26}\sum_{r,s=-\infty}^{\infty}q^{98r^2+294s^2+49r-161s},
	\end{align*}
	\begin{align*}
		T_{7,5}(S_4)=q^{26}\sum_{r,s=-\infty}^{\infty}q^{98r^2+294s^2+49r-161s}.
	\end{align*}
	Hence 5-components cancel in pairs and we get $e_2(7n+5)=0$. Since the proof is similar, we skip the proof of remaining identities. This completes the proof of Theorem  \ref{t:v.1.2}.\\
\end{proof} 
\begin{theorem}\label{t:v.1.3}
	For non-negative integers $u,v$ and $w\leq 15$, define 
	\begin{align*}
		\sum_{n=0}^{\infty}h_u(n)q^n=\left(q^1,q^6;q^7\right)_\infty \left(\pm q^u,\pm q^{21-u};q^{21}\right)_\infty^2,
	\end{align*} 
	\begin{align*}
		\sum_{n=0}^{\infty}i_v(n)q^n=\left(q^2,q^5;q^7\right)_\infty \left(\pm q^v,\pm q^{21-v};q^{21}\right)_\infty^2,
	\end{align*} 
	\begin{align*}
		\sum_{n=0}^{\infty}j_w(n)q^n=\left(q^1,q^6;q^7\right)_\infty \left(\pm q^w,\pm q^{21-w};q^{21}\right)_\infty^2.
	\end{align*}
	\rm{Then}\\  $h_3(7n+6)=h_4(7n+2)=h_{10}(7n+4)=0,$\\
	$i_1(7n+2)=i_6(7n+3)=i_8(7n+6)=0$,\\
	$j_2(7n+4)=j_5(7n+6)=j_9(7n+5)=0.$
\end{theorem} 
\begin{proof}[\it{Proof}] We start with\\
	$\displaystyle \sum_{n=0}^{\infty}h_3(n)q^n=\left(q,q^6;q^7\right)_\infty \left(q^3,q^{18};q^{21}\right)_\infty^2$\\
	$=\dfrac{f\left(-q,-q^6\right)f\left(-q^3,-q^{18}\right)^2}{\left(q^7;q^7\right)_\infty\left(q^{21};q^{21}\right)_\infty^2}$\\	$=\dfrac{\left(f\left(q^9,q^{19}\right)-qf\left(q^5,q^{23}\right)\right)\left(f\left(q^{6},q^{36}\right)\phi(q^{21})-2q^3f\left(q^{15},q^{27}\right)\psi(q^{42})\right)}{\left(q^7;q^7\right)_\infty\left(q^{21};q^{21}\right)_\infty^2}$\\
	$=\dfrac{\phi(q^{21})\left(f\left(q^9,q^{19}\right)f\left(q^6,q^{36}\right)-qf\left(q^5,q^{23}\right)f\left(q^{6},q^{36}\right)\right)}{\left(q^7;q^7\right)_\infty\left(q^{21};q^{21}\right)_\infty^2}\\
	+\dfrac{2\psi(q^{42})\left(-q^3f\left(q^9,q^{19}\right)f\left(q^{15},q^{27}\right)+q^4f\left(q^5,q^{23}\right)f\left(q^{15},q^{27}\right)\right)}{\left(q^7;q^7\right)_\infty\left(q^{21};q^{21}\right)_\infty^2}$\\
	$=\dfrac{\phi(q^{21})(S_1-S_2)}{\left(q^7;q^7\right)_\infty\left(q^{21};q^{21}\right)_\infty^2}\;+\;\dfrac{2\psi(q^{42})(-S_3+S_4)}{\left(q^7;q^7\right)_\infty\left(q^{21};q^{21}\right)_\infty^2}$,\\
	where \begin{align*}
		S_1 = f(q^9,q^{19})f(q^6,q^{36})=	\sum_{m,n=-\infty}^{\infty}q^{14m^2+5m+21n^2+15n},
	\end{align*}
	\begin{align*}
		S_2= qf(q^5,q^{23})f(q^6,q^{36})=\sum_{m,n=-\infty}^{\infty}q^{14m^2+9m+21n^2+15n+1},
	\end{align*}
	\begin{align*}
		S_3= q^3f(q^9,q^{19})f(q^{15},q^{27})=\sum_{m,n=-\infty}^{\infty}q^{14m^2+5m+21n^2+6n+3},
	\end{align*}	
	\begin{align*}
		S_4= q^4f(q^5,q^{23})f(q^{15},q^{27})=\sum_{m,n=-\infty}^{\infty}q^{14m^2+9m+21n^2+6n+4}	.
	\end{align*}

	\noindent	In $S_1$, $5m+15n\equiv6\pmod7$ implies $-m-3n\equiv3\pmod7$ and $-2m+n\equiv-1\pmod7$.By solving these two congruence we get, $m=-r-3s$ and $n=-2r+s-1$. Therefore, 6-component of $S_1$ is \begin{align*}
		T_{7,6}(S_1)=q^6\sum_{r,s=-\infty}^{\infty}q^{98r^2+147s^2+49r-42s}.
	\end{align*}
	Similarly, we obtain
	\begin{align*}
		T_{7,6}(S_2)=q^6\sum_{r,s=-\infty}^{\infty}q^{98r^2+147s^2-49r-42s},
	\end{align*} 
	\begin{align*}
		T_{7,6}(S_3)=	q^{27}\sum_{r,s=-\infty}^{\infty}q^{147r^2+98s^2+105r+49s},
	\end{align*}
	\begin{align*}
		T_{7,6}(S_4)=q^{27}\sum_{r,s=-\infty}^{\infty}q^{147r^2+98s^2+105r-49s}.
	\end{align*}
	Hence 6-components cancel in pairs and we get $h_3(7n+6)=0$. With similar approach one can prove remaining results of Theorem  \ref{t:v.1.3}.
\end{proof}
\begin{remark}
		We note the identities, 	$h_3(7n+6)=i_1(7n+2)=0$, are also special cases of \eqref{4} with $\ell=1$ and $t=1,3$ respectively.
\end{remark}
\begin{theorem}\label{t:v.1.4}
	Define
	\begin{align*}
		\sum_{n=0}^{\infty}k(n)q^n=\left(q^1,q^6;q^7\right)_\infty \left(-q^9,-q^{12};q^{21}\right)_\infty,
	\end{align*} 
	\begin{align*}
		\sum_{n=0}^{\infty}l(n)q^n=\left(q^2,q^5;q^7\right)_\infty \left(-q^3,-q^{18};q^{21}\right)_\infty,
	\end{align*} 
	\begin{align*}
		\sum_{n=0}^{\infty}t(n)q^n=\left(q^3,q^4;q^7\right)_\infty \left(-q^6,-q^{15};q^{21}\right)_\infty,
	\end{align*}
	\begin{align*}
		\sum_{n=0}^{\infty}o(n)q^n=\left(\pm q,\pm q^6;q^7\right)_\infty^2 \left(q^6,q^{15};q^{21}\right)_\infty,
	\end{align*}  
	\begin{align*}
		\sum_{n=0}^{\infty}p(n)q^n=\left(\pm q^2,\pm q^5;q^7\right)_\infty^2 \left(q^9,q^{12};q^{21}\right)_\infty,
	\end{align*}  
	\begin{align*}
		\sum_{n=0}^{\infty}z(n)q^n=\left(\pm q^3,\pm q^4;q^7\right)_\infty^2 \left(q^3,q^{18};q^{21}\right)_\infty.
	\end{align*}
	\rm{Then}   $$k(7n+4)=l(7n+6)=t(7n+5)=o(7n+4)=p(7n+1)=z(7n+5)=0.$$ \end{theorem}
\begin{remark}
	We exclude the proof of Theorem \ref{t:v.1.4} as it is similar to the proof of Theorem \ref{t:v.1.2} and \ref{t:v.1.3}. We also note the identities, $o(7n+4)=p(7n+1)=z(7n+5)=0$ are special cases of \eqref{5} with $\ell=1$ and $t=1,2,3$ respectively.
\end{remark}
\section{Conclusion}
In this paper, we have studied several vanishing coefficients in arithmetic progressions of infinite products defined as in Theorem \ref{t:v.1.1} - \ref{t:v.1.4} with modulo 5 and 7 by employing properties of Ramanujan's theta functions and elementary q series manipulations. It will be also interesting to study the periodicity of coefficients in the series expansions of infinite products defined in \ref{t:v.1.1} - \ref{t:v.1.4}.

\end{document}